\documentclass[1pt,reqno]{amsart}
\usepackage{mathrsfs}
\usepackage{amssymb}
\usepackage{amsmath}
\usepackage{amsfonts}
\usepackage{amscd}
\usepackage{graphicx}
\usepackage[colorlinks, linkcolor=blue, anchorcolor=blue, citecolor=blue]{hyperref}
\begin{document}
\setlength{\oddsidemargin}{0cm} \setlength{\evensidemargin}{0cm}
\baselineskip=20pt

\theoremstyle{plain} \makeatletter
\newtheorem{theorem}{Theorem}[section]
\newtheorem{proposition}[theorem]{Proposition}
\newtheorem{lemma}[theorem]{Lemma}
\newtheorem{coro}[theorem]{Corollary}

\theoremstyle{definition}
\newtheorem{defi}[theorem]{Definition}
\newtheorem{notation}[theorem]{Notation}
\newtheorem{exam}[theorem]{Example}
\newtheorem{prop}[theorem]{Proposition}
\newtheorem{conj}[theorem]{Conjecture}
\newtheorem{prob}[theorem]{Problem}
\newtheorem{remark}[theorem]{Remark}
\newtheorem{claim}{Claim}

\newcommand{\SO}{{\mathrm S}{\mathrm O}}
\newcommand{\SU}{{\mathrm S}{\mathrm U}}
\newcommand{\U}{{\mathrm U}}
\newcommand{\C}{{\mathbb C}}
\newcommand{\Sp}{{\mathrm S}{\mathrm p}}
\newcommand{\so}{{\mathfrak s}{\mathfrak o}}
\newcommand{\Ad}{{\mathrm A}{\mathrm d}}
\newcommand{\ad}{{\mathrm a}{\mathrm d}}
\newcommand{\m}{{\mathfrak m}}
\newcommand{\g}{{\mathfrak g}}
\newcommand{\p}{{\mathfrak p}}
\newcommand{\fk}{{\mathfrak k}}
\newcommand{\fh}{{\mathfrak h}}
\newcommand{\E}{{\mathrm E}}
\newcommand{\F}{{\mathrm F}}
\newcommand{\ffb}{{\mathfrak b}}
\newcommand{\fl}{{\mathfrak l}}
\newcommand{\G}{{\mathrm G}}
\newcommand{\R}{\mathrm{Ric}}
\newcommand{\yd}{\approx}

\numberwithin{equation}{section}

\title{Geodesic orbit metrics on compact simple Lie groups arising from generalized flag manifolds}
\author{Huibin Chen, Zhiqi Chen and Joseph A. Wolf}
\address{School of Mathematical Sciences and LPMC, Nankai University, Tianjin 300071, P.R. China}
\email{chenhuibin@mail.nankai.edu.cn}
\address{School of Mathematical Sciences and LPMC, Nankai University, Tianjin 300071, P.R. China}
\email{chenzhiqi@nankai.edu.cn}
\address{Department of Mathematics University of California Berkeley, CA 94720-3840, USA}
\email{jawolf@math.berkeley.edu}
\date{}
\maketitle

\begin{abstract}
In this paper, we investigate left-invariant geodesic orbit metrics on connected simple Lie groups, where the metrics are formed by the structures of generalized flag manifolds. We prove that all these left-invariant geodesic orbit metrics on simple Lie groups are naturally reductive.
\end{abstract}

\section{Introduction}
For a homogeneous Riemannian manifold $(M=G/H, g)$, where $H$ is a compact subgroup of $G$, $g$ is a $G$-invariant Riemannian metric on $M$. If any geodesic of $M$ is the orbit of some 1-parameter subgroup of $G$, then $M$ is called a geodesic orbit space (g.o. space) and the metric $g$ is called a geodesic orbit metric (g.o. metric). A complete Riemannian manifold $(M, g)$ is called geodesic orbit if it is a geodesic orbit space with respect to its whole connected isometry group. This terminology was introduced by O. Kowalski and L. Vanhecke in \cite{KoVa}, where they started a systematic research of geodesic orbit manifolds and gave the classification results with dimension up to 6. 

After that many mathematicians obtained the classification results with some special settings, the authors can refer to \cite{Ni3}, \cite{Ta}, \cite{ArWaGu} and the reference therein for more information.

In \cite{Ni}, Nikonorov started to investigate g.o. metrics on compact simple Lie groups $G$ with isometry group $G\times K$ where $K$ is a compact subgroup of $G$ and he obtained an equivalent algebraic condition for g.o. spaces. In \cite{ChChDe}, the three authors showed that all the g.o. metrics on compact Lie groups arising from generalized Wallach spaces are naturally reductive. 

In this paper, we investigate all the geodesic orbit metrics on compact simple Lie groups $G$ with the structure from generalized flag manifolds. By using the structure of generalized, we prove that all these g.o. metrics are naturally reductive with respect to $G\times K$.

This paper is organized as follows: in section 2, we introduce the definition and structure of generalized flag manifolds along with some basic facts on g.o. metrics on compact simple Lie groups. In section 3, we prove all these g.o. metrics are naturally reductive by using the structure of generalized flag manifolds.

\section{Geodesic orbit metrics on compact simple Lie groups and generalized flag manifolds}
We first recall some basic conceptions. Let $K$ be a closed subgroup of Lie group $G$, a $G$-invariant metric $g$ on $M=G/K$ corresponds to an $Ad(K)$-invariant scalar product $(\ ,\ )$ on $\m=T_oM$ and vice versa. The metric $g$ is called standard if the scalar product $(\ ,\ )$ on $\m$ is the restriction of $B$, where $B$ is the minus of Killing form of $\g$. For a given non-degenerate $Ad(K)$-invariant scalar product $(\ ,\ )$ on $\m$, there exist an $Ad(K)$-invariant positive definite symmetric operator $A$ on $\m$ such that $(x, y)=B(Ax, y) (x, y\in\m)$. Conversely, any such operator $A$ determines an $Ad(K)$-invariant scalar product $(x,y)= B(Ax,y)$ on $\m$. We call such $A$ a metric endomorphism. A homogeneous Riemannian metric on $M=G/K$ is called naturally reductive if 
$$([Z,X]_{\m},Y)+(X,[Z,Y]_{\m})=0, \forall X,Y,Z\in\m.$$

In \cite{AlAr}, there is an equivalent algebraic description of g.o. metrics on $M=G/K$, we recall it below:
\begin{theorem}[\cite{AlAr} Corollary 2]\label{g.o.space}
Let $(M = G/K,g)$ be a homogeneous Riemannian manifold. Then $M$ is geodesic orbit space if and only if for every $X\in\m$ there exists an $a(X) \in \fk$ such that
$$[a(X) + X, AX]\in\fk,$$
where $A$ is the metric endomorphism.
\end{theorem}

According to the Ochiai-Takahashi theorem \cite{OcTa}, the full connected isometry group Isom($G,g$) of a simple compact Lie group $G$ with a left-invariant Riemannian metric $g$ contains in the group $L(G)R(G)$, the product of left and right translations. Hence $G$ is a normal subgroup in Isom($G, g$), which is locally isomorphic to the group $G\times K$, where $K$ is a closed subgroup of $G$, with action $(a, b)(c) = acb^{-1}$, where $a,c \in G$ and $b \in K$.

In \cite{AlNi}, Alekseevski and Nikonorov showed that if we choose $G$ as the isometry group of the compact Lie group $G$ with a left-invariant Riemannian metric, then 
\begin{prop}[\cite{AlNi} Proposition 8]
A compact Lie group $G$ with a left-invariant metric $g$ is a g.o. space if and only if the corresponding Euclidean metric $(\ ,\ )$ on the Lie algebra $\g$ is bi-invariant.
\end{prop}

In \cite{Ni}, Nikonorov consider the isometry group of compact simple Lie group $G$ as $G\times K$, where $K$ is a closed subgroup of $G$. Then he obtained the equivalent algebraic description of g.o. metrics $g$ on compact simple Lie groups $G$:
\begin{theorem}[\cite{Ni} Proposition 10]\label{g.o.group}
A simple compact Lie group $G$ with a left-invariant Riemannian metric $g$ is a geodesic orbit manifold if and only if there is a closed connected subgroup $K$ of $G$ such that for any $X\in\g$ there is $W\in\fk$ such that for any $Y\in\g$ the equality $([X + W, Y ], X) = 0$ holds or, equivalently, $[A(X), X + W ] = 0$, where $A:\g\rightarrow\g$ is a metric endomorphism.
\end{theorem} 

Let $B$ denote the minus of Killing form of $\g$, the Lie algebra of $G$. 
\begin{equation}\label{metric}
(\ ,\ )=A_0B(\ ,\ )|_{\fk_0}+x_1B(\ ,\ )|_{\fk_1}+\cdots+x_pB(\ ,\ )|_{\fk_p}+y_1B(\ ,\ )|_{\m_1}+\cdots+y_qB(\ ,\ )|_{\m_q},
\end{equation}
where $\fk$ is the Lie algebra of $K$ and $\fk=\fk_0\oplus\fk_1\oplus\cdots\fk_p$ is the decomposition of $\fk$ into non-isomorphic simple ideals and center, $\m$ is the $B$-orthogonal component of $\fk$ and $\m=\m_1\oplus\cdots\oplus\m_q$ is the decomposition of $\m$ into $Ad(K)$-irreducible and non-equivalent modules. We denote $A$ as $diag\{x_1,\cdots,x_p,y_1,\cdots,y_q\}$, $A$ is clearly the metric endomorphism.

D' Atri and Ziller \cite{DAZi} have investigated naturally reductive metrics among the left-invariant metrics on compact Lie groups, and have given a complete classification in the case of simple Lie groups. The following is a description of naturally reductive left-invariant metrics on a compact simple Lie group:
\begin{theorem}[\cite{DAZi} Theorem 1, Theorem 3]\label{nr}
Under the notations above, a left-invariant metric on $G$ of the form
\begin{equation}\label{form}
(\ ,\ )=xB|_{\m}+A_0|_{\fk_0}+u_1B|_{\fk_1}+\cdots+u_p·B|_{\fk_p}, (x,u_1,\cdots,u_p\in\mathbb{R}^+) 
\end{equation}
is naturally reductive with respect to $G\times K$, where $G\times K$ acts on $G$ by $(g, k)y = gyk^{-1}$, where $A_0$ is an arbitrary metric on $\fk_0$. Conversely, if a left-invariant metric $(\ ,\ )$ on a compact simple Lie group $G$ is naturally reductive, then there exists a closed subgroup $K$ of $G$ such that the metric $(\ ,\ )$ is given by the form (\ref{form}).
\end{theorem}

We have the following corollary:
\begin{coro}\label{ns}
Let $g$ of the form (\ref{metric}) be a non-naturally reductive g.o. metric on compact Lie group $G$ and let $\tilde{g}$ be the restriction of $g$ on $\m$, denote the corresponding metric endomorphism by $A$ and $\tilde{A}$, respectively. Then $(M=G/K,\tilde{g})$ is a g.o. metric on $M$ not homothetic to the standard metric.
\end{coro}
\begin{proof}
Since $g$ is a g.o. metric on $G$, then by Theorem \ref{g.o.group}, we have for any $X\in\m$, there exists $W\in\fk$ such that
$$[W+X,A(X)]=[W+X,\tilde{A}(X)]=0\in\fk,$$
by Theorem \ref{g.o.space}, $(M=G/K,\tilde{g})$ is a g.o. space. From Theorem \ref{nr}, we know $\tilde{g}$ is not homothetic to the standard metric because $g$ is non-naturally reductive.
\end{proof}

Next, we will introduce some basic conceptions on generalized flag manifold.
\begin{defi}[\cite{Ar}]
A generalized flag manifold is a homogeneous space of the form $G/K=G/C(S)$, where $G$ is a compact Lie group and $S$ is a torus in $G$. If the torus $S$ is a maximal torus in $G$, say $T$, then $G/T$ is called a flag manifold.
\end{defi}

Let $G/K = G/C(S)$ be a generalized flag manifold, where $G$ is a compact semisimple Lie group and $S$ is a torus in $G$, here $C(S)$ denotes the centralizer of $S$ in $G$. Let $\g$ and $\fk$ be the Lie algebras of the Lie groups $G$ and $K$ respectively, and $\g^\mathbb{C}$ and $\fk^\mathbb{C}$ be the complexifications of $\g$ and $\fk$ respectively. Let $\g = \fk\oplus\m$ be a reductive decomposition with respect to $B$ with $[\fk,\m]\subset\m$. Let $T$ be a maximal torus containing $S$. Then this is a maximal torus in $K$. Let $\mathfrak{h}$ be the Lie algebra of $T$ and $\mathfrak{h}^\mathbb{C}$ its complex. Then $\mathfrak{h}^\mathbb{C}$ is a Cartan subalgebra of $\g^\mathbb{C}$. Let $R$ be a root system $\g^\mathbb{C}$ with respect to $\mathfrak{h}^\mathbb{C}$ and $\g^\C = \mathfrak{h}^\C \oplus  \sum_{\alpha\in R}\g_{\alpha}^\C$ be the root space decomposition.

Obviously, $\fk^\C$ contains $\mathfrak{h}^\C$, so there exist a subset $R_K$ of $R$ such that $\fk^\C=\mathfrak{h}^\C+\sum_{\alpha\in R_K}\g_{\alpha}^\C$. We can choose $\Pi$ and $\Pi_K$ to be simple roots of $R$ and $R_K$ respectively such that $\Pi_K\subset\Pi$. Let $R_M=R\setminus R_K$, then we have $\m^\C=\sum_{\alpha\in R_M}\g_\alpha^\C$ and
$$\g^\C=\mathfrak{h}^\C\oplus\sum_{\alpha\in R_K}\g_{\alpha}^\C\oplus\sum_{\alpha\in R_M}\g_{\alpha}^\C.$$

We choose a Weyl basis $\{H_{\alpha}, E_{\alpha}|\alpha\in R\}$ in $\g^\C$ with $B(E_\alpha,E_{-\alpha})=1$, $[E_\alpha, E_{-\alpha}]=H_\alpha$ and 
$$ [E_\alpha, E_\beta]=\left\{
\begin{aligned}
&0  &\mbox{if $\alpha+\beta\notin R$ and $\alpha+\beta\neq0$} \\
N_{\alpha,\beta}&E_{\alpha+\beta}   &\mbox{if $\alpha+\beta\in R,$}
\end{aligned}
\right.
$$
where $N_{\alpha,\beta}(\neq0)$ is the structure constant with $N_{\alpha,\beta}=-N_{-\alpha,-\beta}$ and $N_{\alpha,\beta}=-N_{\beta,\alpha}.$ The following is a compact real form of $\g^\C$:
$$\g_\mu=\sum_{\alpha\in R^+}\mathbb{R}\sqrt{-1}H_\alpha\oplus\sum_{\alpha\in R^+}(\mathbb{R}A_{\alpha}+\mathbb{R}B_\alpha),$$
where $R^+$ is the positive root system of $\g$ and $A_\alpha=E_\alpha-E_{-\alpha}, B_\alpha=\sqrt{-1}(E_\alpha+E_{-\alpha}).$ Since any two compact real forms of $\g^\C$ are conjugated, we can identify $\g$ with $\g_\mu$. If we set $R_M^+=R^+\setminus R_K^+$, then we have 
$$\fk=\sum_{\alpha\in R^+}\mathbb{R}\sqrt{-1}H_\alpha\oplus\sum_{\alpha\in R_K^+}(\mathbb{R}A_\alpha+\mathbb{R}B_{\alpha}),$$
$$\m=\sum_{\alpha\in R_M^+}(\mathbb{R}A_\alpha+\mathbb{R}B_{\alpha}).$$

The next lemma shows the bracket computation of $\g$ which we will make much use of in the proof of our main theorem.
\begin{lemma}\label{structure}
The Lie bracket among the elements of $\{A_\alpha=E_\alpha-E_{-\alpha}, B_\alpha=\sqrt{-1}(E_\alpha+E_{-\alpha}), \sqrt{-1}H_\beta\in R^+, \beta\in\Pi\}$ of $\g$ are given by
$$[ \sqrt{-1}H_\alpha,A_\beta] = \beta(H_\alpha)B_\beta,\quad [A_\alpha, A_\beta] = N_{\alpha,\beta}A_{\alpha+\beta} + N_{-\alpha,\beta}A_{\alpha-\beta}(\alpha\neq\beta),$$
$$[\sqrt{-1}H_\alpha,B_\beta]=-\beta(H_\alpha)A_\beta, [B_\alpha,B_\beta]=-N_{\alpha,\beta}A_{\alpha+\beta} -N_{\alpha,-\beta}A_{\alpha-\beta}(\alpha\neq\beta),$$
$$[A_\alpha,B_\alpha] = 2 \sqrt{-1}H_\alpha, [A_\alpha,B_\beta] = N_{\alpha,\beta}B_{\alpha+\beta} + N_{\alpha,-\beta}B_{\alpha-\beta}(\alpha\neq\beta),$$
where $N_{\alpha,\beta}$ are the structural constants in Weyl basis.
\end{lemma}

In generalized flag manifolds, the so-called $\frak{t}$-roots play an very important role which we will introduce in the following.

From now on we fix a system of simple roots $\Pi= \{\alpha_1,\cdots,\alpha_r,\phi_1,\cdots,\phi_k\}$ of $R$, so that $\Pi_K =\{\phi_1,\cdots,\phi_k\}$ is a basis of the root system $R_K$ and $\Pi_M =\Pi\setminus\Pi_K = \{\alpha_1,\cdots,\alpha_r\} (r + k = l)$. Let $\{h_{\alpha_1},\cdots,h_{\alpha_r},h_{\phi_1},\cdots,h_{\phi_k}\}$ be the fundamental weights. Let 
$$\frak{t}=\frak{z}(\fk^\C)\cap\sqrt{-1}\frak{h},$$
where $\frak{z}(\fk^\C)$ is the center of $\fk^\C$. Consider the restriction map $\pi:(\frak{h}^\C)^*\rightarrow \frak{t}^*$ defined by $\pi(\alpha) = \alpha |_{\frak{t}}$, and set $R_{\frak{t}} = \pi(R) = \pi(R_M)$. $\frak{t}$-roots are the elements of $R_{\frak t}$. For an invariant ordering $R_M^+ =R^+\setminus R_K^+$ in $R_M$, we set $R_{\frak t}^+ =\pi(R_M^+)$ and $R_{\frak t}^- =-R_{\frak t}^+$. It is obvious that $R_{\frak t}^-= \pi(R_M^- )$, thus the splitting $R_{\frak t} = R_{\frak t}^-\cup R_{\frak t}^+$ defines an ordering in $R_{\frak t}$. A $\frak t$-root $\xi\in R_{\frak t}^+$ (respectively $\xi\in R_{\frak t}^-$) will be called positive (respectively negative). A $\frak t$-root is called simple if it is not a sum of two positive $\frak t$-roots.

\begin{theorem}[\cite{AlPe} Corollary 3.1]
There is one-to-one correspondence between $\frak t$-roots and complex irreducible $ad(\fk^\C)$-submodules $\m_\xi$ of $\m^\C$. This correspondence is given by
$$R_{\frak t}\ni\xi\leftrightarrow \m_\xi=\sum_{\alpha\in R_M, \pi(alpha)=\xi}\C E_\alpha,$$
Hence $\m^\C=\sum_{\xi\in R_{\frak t}}\m_\xi$. Moreover, these submodules are non-equivalent $ad(\fk^\C)$-modules.
\end{theorem}

Since the complex conjugation $\tau:\g^\C\rightarrow \g^\C$ with respect to the compact real form $\g$ interchanges the root spaces, a decomposition of the real $ad(\fk)$-module $\m=(\m^\C)^\tau$ into real irreducible $ad(\fk)$-submodule is given by
\begin{equation}\label{m}
\m=\sum_{\xi\in R_{\frak t}^+} (\m_\xi\oplus\m_{-\xi})^\tau,
\end{equation}
where $V^\tau$ denotes the set of fixed points of the complex conjugation $\tau$ in a vector subspace $V\subset \g^\C$. If we set $R_{\frak t}^+ =\{\xi_1,\cdots,\xi_s\}$, then according to (\ref{m}) each real irreducible $ad(\fk)$-submodule $\m_i = (\m_{\xi_i} \oplus \m_{-\xi_i} )^\tau (1\leq i\leq s)$ corresponding to the positive $\frak t$-roots $\xi_i$, is given by
$$\m_i=\sum_{\alpha\in R_M^+, \pi(\alpha)=\xi_i}(\mathbb{R}A_\alpha+\mathbb{R}B_\alpha).$$

\section{Main theorem and its proof}
In this section, we will claim our main theorem and prove it.
\begin{theorem}\label{maintheorem}
All the g.o. metrics on compact simple Lie groups $G$ of the form (\ref{metric}) arising from generalized flag manifold are naturally reductive.
\end{theorem}

In \cite{AlAr}, the authors investigated all g.o. metrics on generalized flag manifolds (they called them flag manifolds in their paper) of compact simple Lie groups and they proved that only $\SO(2l+1)/\U(l)(l\geq2)$ and $\Sp(l)/\U(1)\Sp(l-1)(l\geq3)$ can admit g.o. metrics not homothetic to the standard metrics. As a result of Corollary \ref{ns}, we only need to consider whether there are non-naturally reductive g.o. metrics on $\SO(2l+1)(l\geq2)$ and $\Sp(l)(l\geq3)$ with the corresponding metric forms. For these two special generalized flag manifolds, the metric for (\ref{metric}) can be simplified as follows:
\begin{equation}\label{metric-s}
(\ ,\ )=B(\ ,\ )|_{\frak{u}(1)}+uB(\ ,\ )|_{\fk_0}+xB(\ ,\ )|_{\m_1}+yB(\ ,\ )|_{\m_2},
\end{equation}
where $\frak{u}(1)$ is a 1-dimensional center of $\fk$ and $\fk_0$ is a simple Lie algebra.

When apply Theorem \ref{g.o.group} to the metric form (\ref{metric-s}), we can immediately obtain the following equivalent description of g.o. metric of this form:
\begin{theorem}\label{thm}
Compact simple Lie group $G$ with the left-invariant metric induced by (\ref{metric-s}) is a geodesic orbit space if and only if for any $T\in\frak{u}(1), H\in\fk_0, X_1\in\m_1, X_2\in\m_2$, there exists $K\in\fk$ such that the following three conditions hold:
\begin{enumerate}
\item $[H,K]=0;$
\item $[(x-1)T+(x-u)H+xK+(x-y)X_2, X_1]=0;$
\item $[(y-1)T+(y-u)H+yK, X_2]=0.$
\end{enumerate}
\end{theorem}

In the following, we will prove all the g.o. metrics of the form (\ref{metric-s}) on $\SO(2l+1)(l\geq2)$ and $\Sp(l)(l\geq3)$ are naturally reductive for each case.

\subsection{Case of $\SO(2l+1)$} The painted Dynkin diagram of this case is\\
\setlength{\unitlength}{1mm}
\begin{picture}(6,6)
\put(10,0){$B_l:$}
\put(20,0){$\circ$}
\put(21.5,1){\line(1,0){11}}
\put(32,0){$\circ$}
\put(33.3,1){\line(1,0){6}}
\put(42,1){$\ldots$}
\put(52,1){$\ldots$}
\put(58.3,1){\line(1,0){6}}
\put(64,0){$\circ$}
\put(65.5,0.7){\line(1,0){10}}
\put(65.5,1.3){\line(1,0){10}}
\put(69,0.25){$>$}
\put(75,0){$\bullet$}
\put(20,-4){$\alpha_1$}
\put(32,-4){$\alpha_2$}
\put(63,-4){$\alpha_{l-1}$}
\put(75,-4){$\alpha_l$}
\end{picture}
\quad\\

Hence we can give the basis for each of the four parts in the decomposition $\frak{so}(2l+1)=\frak{u}(1)\oplus\frak{su}(l)\oplus\m_1\oplus\m_2.$\\
$\frak{u}(1)=Span_{\mathbb R}\{\sqrt{-1}H_{\alpha_l}\},$\\
$\frak{su}(l)=Span_{\mathbb R}\{A_\alpha, B_\alpha, \sqrt{-1}H_\beta|\alpha=\alpha_p+\cdots+\alpha_k,1\leq p\leq k\leq l-1;\beta=\alpha_p,1\leq p\leq l-1\},$\\
$\m_1=Span_{\mathbb R}\{A_\alpha, B_\alpha|\alpha=\alpha_k+\cdots+\alpha_{l-1}+\alpha_{l},1\leq k\leq l\},$\\
$\m_2=Span_{\mathbb R}\{A_\alpha, B_\alpha|\alpha=\alpha_k+\cdots+2\alpha_p+\cdots+2\alpha_l,1\leq k< p\leq l\}.$

Then we choose $T=\sqrt{-1}H_{\alpha_l}, H=\sum_{i=1}^{l-1}\sqrt{-1}H_{\alpha_i}, X_1=B_{\alpha_l}, X_2=A_{\alpha_1+\cdots+\alpha_{l-1}+2\alpha_l}$ and we assume the metric of the form (\ref{metric-s}) is a g.o. metric, by Theorem \ref{thm}, there exists some $K\in\fk$ such that
\begin{enumerate}
\item $[H,K]=0;$
\item $[(x-1)T+(x-u)H+xK+(x-y)X_2, X_1]=0;$
\item $[(y-1)T+(y-u)H+yK, X_2]=0.$
\end{enumerate}
From (2), we have $[(x-1)\sqrt{-1}H_{\alpha_l}+(x-u)\sum_{i=1}^{l-1}\sqrt{-1}H_{\alpha_i}+xK,B_{\alpha_l}]=(y-x)[A_{\alpha_1+\cdots+\alpha_{l-1}+2\alpha_l},B_{\alpha_l}]$. 

By Lemma \ref{structure}, we have 
$$[(x-1)\sqrt{-1}H_{\alpha_l}+(x-u)\sum_{i=1}^{l-1}\sqrt{-1}H_{\alpha_i}+xK,B_{\alpha_l}]=(y-x)N_{\alpha_1+\cdots+\alpha_{l-1}+2\alpha_l,-\alpha_l}A_{\alpha_1+\cdots+\alpha_{l-1}+\alpha_l}.$$

We next prove that there is no $A_{\alpha_1+\cdots+\alpha_{l-1}+\alpha_l}$-component in $[(x-1)\sqrt{-1}H_{\alpha_l}+(x-u)\sum_{i=1}^{l-1}\sqrt{-1}H_{\alpha_i}+xK,B_{\alpha_l}]$, in fact, we only need to show $K$ doesn't contain $B_{\alpha_1+\cdots+\alpha_{l-1}}$-component by Lemma \ref{structure}. If $K$ contains $B_{\alpha_1+\cdots+\alpha_{l-1}}$-component, then
\begin{align}
[\sum_{i=1}^{l-1}\sqrt{-1}H_{\alpha_i},B_{\alpha_1+\cdots+\alpha_{l-1}}]&=-\sum_{i=1}^{l-1}(\alpha_1+\cdots+\alpha_{l-1})(H_{\alpha_i})A_{\alpha_1+\cdots+\alpha_{l-1}}\\
&=-\sum_{i=1}^{l-1}<\alpha_1+\cdots+\alpha_{l-1},\alpha_i>A_{\alpha_1+\cdots+\alpha_{l-1}}
\end{align}
From $B_l$'s Cartan matrix, we know $[\sum_{i=1}^{l-1}\sqrt{-1}H_{\alpha_i},B_{\alpha_1+\cdots+\alpha_{l-1}}]\neq0$, which is a contradiction to (1) above. As a result, there is no $A_{\alpha_1+\cdots+\alpha_{l-1}+\alpha_l}$-component in $[(x-1)\sqrt{-1}H_{\alpha_l}+(x-u)\sum_{i=1}^{l-1}\sqrt{-1}H_{\alpha_i}+xK,B_{\alpha_l}]$. Hence, $x=y$. By Theorem \ref{nr}, g.o. metrics on $\SO(2l+1)$ of the form (\ref{metric-s}) are naturally reductive with respect to $\SO(2l+1)\times\U(l)$.

\subsection{Case of $\Sp(l)$} The painted Dynkin diagram of this case is \\
\setlength{\unitlength}{1mm}
\begin{picture}(6,6)
\put(10,0){$C_l:$}
\put(20,0){$\bullet$}
\put(21.5,1){\line(1,0){10.6}}
\put(32,0){$\circ$}
\put(33.5,1){\line(1,0){6}}
\put(42,1){$\ldots$}
\put(52,1){$\ldots$}
\put(58,1){\line(1,0){6}}
\put(64,0){$\circ$}
\put(65.3,0.7){\line(1,0){10}}
\put(65.3,1.3){\line(1,0){10}}
\put(69,0.25){$<$}
\put(75,0){$\circ$}
\put(20,-4){$\alpha_1$}
\put(32,-4){$\alpha_2$}
\put(63,-4){$\alpha_{l-1}$}
\put(75,-4){$\alpha_l$}
\end{picture}
\quad\\

The basis of each part of the decomposition $\frak{sp}(l)=\frak{u}(1)\oplus\frak{sp}(l-1)\oplus\m_1\oplus\m_2$ are as follows:\\
$\frak{u}(1)=Span_{\mathbb R}\{\sqrt{-1}H_{\alpha_1}\},$\\
$\frak{sp}(l-1)=Span_{\mathbb R}\{A_\alpha, B_\alpha, \sqrt{-1}H_\beta|\beta=\alpha_i(2\leq i\leq l); \alpha=\alpha_p+\cdots+\alpha_k(2\leq p\leq k\leq l)\mbox{ or }\alpha=\alpha_p+\alpha_{p+1}+\cdots+2\alpha_k+\cdots+2\alpha_{l-1}+\alpha_l(2\leq p\leq k\leq l-1)\},$\\
$\m_1=Span_{\mathbb{R}}\{A_\alpha, B_\alpha|\alpha=\alpha_1+\cdots+\alpha_k(1\leq k\leq l)\mbox{ or }\alpha=\alpha_1+\alpha_2+\cdots+2\alpha_p+\cdots+2\alpha_{l-1}+\alpha_l(2\leq p\leq l-1)\},$\\
$\m_2=Span_{\mathbb R}\{A_{2\alpha_1+\cdots+2\alpha_{l-1}+\alpha_l},B_{2\alpha_1+\cdots+2\alpha_{l-1}+\alpha_l}\}.$

We assume the metric of the form (\ref{metric-s}) on $\Sp(l)$ is a g.o. metric, then for $T=\sqrt{-1}H_{\alpha_1}, H=\sum_{i=2}^{l}\sqrt{-1}H_{\alpha_i}, X_1=B_{\alpha_1}, X_2=A_{2\alpha_1+\cdots+2\alpha_{l-1}+\alpha_l}$, by Theorem \ref{thm}, there exists some $K\in\fk$ such that
\begin{enumerate}
\item $[H,K]=0;$
\item $[(x-1)T+(x-u)H+xK+(x-y)X_2, X_1]=0;$
\item $[(y-1)T+(y-u)H+yK, X_2]=0.$
\end{enumerate}
From (2) above, we have 
$[(x-1)\sqrt{-1}H_{\alpha_1}+(x-u)\sum_{i=2}^{l}\sqrt{-1}H_{\alpha_i}+xK, B_{\alpha_1}]=(y-x)[A_{2\alpha_1+\cdots+2\alpha_{l-1}+\alpha_l},B_{\alpha_1}]$. 

By Lemma \ref{structure}, we have 
$$[(x-1)\sqrt{-1}H_{\alpha_1}+(x-u)\sum_{i=2}^{l}\sqrt{-1}H_{\alpha_i}+xK,B_{\alpha_1}]=(y-x)N_{2\alpha_1+\cdots+2\alpha_{l-1}+\alpha_l,-\alpha_1}A_{\alpha_1+2\alpha_2+\cdots+2\alpha_{l-1}+\alpha_l}.$$

We next prove that there is no $A_{\alpha_1+2\alpha_2+\cdots+2\alpha_{l-1}+\alpha_l}$-component in $[(x-1)\sqrt{-1}H_{\alpha_1}+(x-u)\sum_{i=2}^{l}\sqrt{-1}H_{\alpha_i}+xK,B_{\alpha_1}]$, in fact, we only need to show $K$ doesn't contain $B_{2\alpha_2+\cdots+2\alpha_{l-1}+\alpha_l}$-component by Lemma \ref{structure}. If $K$ contains $B_{2\alpha_2+\cdots+2\alpha_{l-1}+\alpha_l}$-component, then
\begin{align}
[\sum_{i=2}^{l}\sqrt{-1}H_{\alpha_i},B_{2\alpha_2+\cdots+2\alpha_{l-1}+\alpha_l}]&=-\sum_{i=2}^{l}(2\alpha_2+\cdots+2\alpha_{l-1}+\alpha_l)(H_{\alpha_i})A_{2\alpha_2+\cdots+2\alpha_{l-1}+\alpha_l}\\
&=-\sum_{i=2}^{l}<2\alpha_2+\cdots+2\alpha_{l-1}+\alpha_l,\alpha_i>A_{2\alpha_2+\cdots+2\alpha_{l-1}+\alpha_l}
\end{align}
From $C_l$'s Cartan matrix, we know $[\sum_{i=2}^{l}\sqrt{-1}H_{\alpha_i},B_{2\alpha_2+\cdots+2\alpha_{l-1}+\alpha_l}]\neq0$, which is a contradiction to (1) above. As a result, there is no $A_{\alpha_1+2\alpha_2+\cdots+2\alpha_{l-1}+\alpha_l}$-component in $[(x-1)\sqrt{-1}H_{\alpha_1}+(x-u)\sum_{i=2}^{l}\sqrt{-1}H_{\alpha_i}+xK,B_{\alpha_1}]$. Hence, $x=y$. By Theorem \ref{nr}, g.o. metrics on $\Sp(l)$ of the form (\ref{metric-s}) are naturally reductive with respect to $\Sp(l)\times(\U(1)\times\Sp(l-1))$.

We complete the proof of Theorem \ref{maintheorem}.
\begin{remark}
For the details of the relationship between painted Dynkin diagrams and generalized flag manifolds, the readers can refer to \cite{Al} and \cite{Ar} for more information.
\end{remark}


\begin{thebibliography}{10}
\bibitem{Al} D. V. Alekseevsky, \emph{Flag manifolds}, Sbornik Radova 11 (1997): 3-35.
\bibitem{AlAr} D. V. Alekseevsky, A. Arvanitoyeorgos, \emph{Riemannian flag manifolds with homogeneous geodesics}, Trans. Amer. Math. Soc., 359 (2007), 3769-3789.
\bibitem{AlNi} D. V. Alekseevsky, Yu.G. Nikonorov, \emph{Compact Riemannian manifolds with homogeneous geodesics}, SIGMA Symmetry Integrability Geom. Methods Appl., 5 (2009), 093, 16 pages.
\bibitem{AlPe} D. V. Alekseevsky, A.M. Perelomov, \emph{Invariant Kähler-Einstein metrics on compact homogeneous spaces}, Funct. Anal. Appl. 20 (1986) 171-182.
\bibitem{Ar} A. Arvanitoyeorgos, \emph{An introduction to Lie groups and the geometry of homogeneous spaces}, Vol. 22. American Mathematical Soc., 2003.
\bibitem{ArWaGu} A. Arvanitoyeorgos, Y. Wang, and G. Zhao. \emph{Riemannian go metrics in certain M-spaces}. Diff. Geom. Appl. 54 (2017): 59-70.
\bibitem{ChChDe} H. Chen, Z. Chen, S. Deng, \emph{Compact simple Lie groups admitting left-invariant Einstein metrics that are not geodesic orbit}, C. R. Math. Acad. Sci. Paris 356 (2018), no. 1, 81-84.
\bibitem{DAZi} J. E. D' Atri and W. Ziller, \emph{Naturally reductive metrics and Einstein metrics on compact Lie groups}, Memoirs Amer. Math. Soc.  19 (215) (1979).
\bibitem{KoVa} O. Kowalski, L. Vanhecke, \emph{Riemannian manifolds with homogeneous geodesics}, Boll. Un. Mat. Ital. B (7), 5(1) (1991), 189-246.
\bibitem{Ni} Y. G. Nikonorov, \emph{On left-invariant Einstein Riemannian metrics that are not geodesic orbit}, Transformation Groups, 2018, DOI:10.1007/s00031-018-9476-7
\bibitem{Ni3} Yu.G. Nikonorov, \emph{Geodesic orbit Riemannian metrics on spheres}, Vladikavkaz. Mat. Zh., 15(3) (2013), 67-76.
\bibitem{OcTa} T. Ochiai, T. Takahashi, \emph{The group of isometries of a left invariant Riemannian metric on a Lie group}, Math. Ann. 223(1) (1976), 91-96.
\bibitem{Ta} H. Tamaru, \emph{Riemannian g.o. spaces fibered over irreducible symmetric spaces}, Osaka J. Math., 36 (1999), 835-851.










\end{thebibliography}
\end{document}